\documentclass[12pt]{amsart}

\usepackage{amssymb}
\usepackage{verbatim}

\newtheorem{thm}{Theorem}[section]

\newtheorem{lem}[thm]{Lemma}
\newtheorem{cor}[thm]{Corollary}

\newtheorem{conj}[thm]{Conjecture}

\theoremstyle{definition}

\theoremstyle{remark}

\numberwithin{equation}{section}

\newcommand{\uple}[1]{\text{\boldmath${#1}$}}

\begin{document}


\title[New error term for the fourth moment]{New error term for the fourth moment of automorphic $L$-functions }


\author{Olga  Balkanova}
\address{Institute for Applied Mathematics of Russian Academy of Sciences,  Khabarovsk, Russia}
\email{olgabalkanova@gmail.com}

\author{Dmitry Frolenkov}
\address{Steklov Mathematical Institute of Russian Academy of Sciences\\
National Research University Higher School of
Economics, Moscow, Russia}
\email{frolenkov@mi.ras.ru}

\begin{abstract}
We improve the error term in the asymptotic formula for the twisted fourth  moment of automorphic $L$-functions of prime level and weight two proved by Kowalski, Michel and Vanderkam. As a consequence, we obtain a new subconvexity bound in the level aspect and improve the lower bound on proportion of simultaneous non-vanishing.
\end{abstract}


\maketitle

\tableofcontents

\section{Introduction}

The fourth moment of automorphic $L$-functions has been studied in  \cite{DFI, KMV} using the large sieve inequality and $\delta$-symbol method.
As an application Duke, Friendlander and Iwaniec proved the subconvexity bound in the level aspect. Another consequence -- simultaneous non-vanishing -- was derived by Kowalski, Michel and Vanderkam.

 In this paper, we optimize several estimates of \cite{KMV}  and compute the explicit dependence of  error terms on the smallest
positive eigenvalue for the Hecke congruence subgroup. This allows us to improve  the results of \cite{DFI, KMV}
 by applying the Kim-Sarnak bound.

We borrow some notations of \cite{DFI, KMV}.  Consider the family $H^{*}_{2}(q)$ of primitive newforms of prime level $q$ and weight $2$.
Every $f \in H_{2}^{*}(q)$ has a Fourier expansion
\begin{equation}
f(z)=\sum_{n\geq 1}\lambda_f(n)n^{1/2}e(nz).
\end{equation}

The associated  $L$-function is defined by
\begin{equation}L(f,s)=\sum_{n \geq 1}\frac{\lambda_f(n)}{n^s}, \quad \Re{s}>1.
\end{equation}
The completed $L$- function
\begin{equation}
\Lambda(f,s)=\left(\frac{\sqrt{q}}{2 \pi}\right)^s\Gamma\left(s+\frac{1}{2}\right)L(f,s)
\end{equation}
can be analytically continued on the whole complex plane. It satisfies the functional equation
\begin{equation} \label{eq: functionalE}
\Lambda(f,s)=\epsilon_f\Lambda(f,1-s), \quad \epsilon_f=\pm1.
\end{equation}
We introduce the natural and harmonic averages
\begin{equation}\label{eq:averages}
\sum_{f \in H_{2}^{*}(q)}^{n}\alpha_f:=\sum_{f \in H_{2}^{*}(q)}\frac{\alpha_f}{|H_{2}^{*}(q)|},
\quad
\sum_{f \in H_{2}^{*}(q)}^{h}\alpha_f:=\sum_{f \in H_{2}^{*}(q)}\frac{\alpha_f}{4\pi \langle f,f\rangle_q},
\end{equation}
where $\langle f,f\rangle_q$ is the Petersson inner product on the space of level $q$ holomorphic modular forms.

The goal of the present paper is to improve the error term in the asymptotic formula for the twisted fourth moment
\begin{equation}
M(l)=\sum_{f \in H_{2}^{*}(q)}^{h}\lambda_f(l)|L(f,1/2+\mu)|^4, \quad \mu \in i\mathbb{R}.
\end{equation}
Our main result is the following.
\begin{thm}\label{mainthm} Let $q$ be a prime and $l<q$. There exists some $B>0$ such that for any $\epsilon>0$
\begin{multline}\label{asymp1}
M(l)=M^D(l)+M^{OD}(l)+M^{OOD}(l)+\\
O_{\epsilon}\left(q^{\epsilon}(1+|\mu|)^B\left(l^{\frac{5-6\theta}{8-8\theta}}q^{-\frac{1-2\theta}{8-8\theta}} +l^{\frac{17}{8}}q^{-\frac{1}{4}}+ l^{\frac{5-4\theta}{8-8\theta}}q^{-\frac{1}{8-8\theta}}\right) \right),
\end{multline}
where  $M^{D}(l)$, $M^{OD}(l)$ and $M^{OOD}(l)$ are the main terms defined by equations $(17)$, $(31)-(32)$ and $(34)$ of \cite{KMV}.
\end{thm}

Here
\begin{equation}\label{eq:theta}
\theta:=\sqrt{\max{(0,1/4-\lambda_1)}}
\end{equation}
 and $\lambda_1=\lambda_1(q)$ is the smallest
positive eigenvalue for the Hecke congruence subgroup $\Gamma_0(q)$.
Currently the best known estimate on $\lambda_1$ is  due to Kim and Sarnak \cite{KS}. Accordingly, we can take
$$\theta=7/64.$$
\begin{cor} Let  $q$ be a prime. For all $\epsilon>0$
\begin{equation}\label{asymp2}
M(1)=P(\log{q})+O_{\epsilon}\left(q^{-25/228+\epsilon}\right),
\end{equation}
where $P$ is a polynomial of degree $6$ and the leading coefficient is $1/60\pi^2$.
\end{cor}

This improves corollary $1.3$ of \cite{KMV}, where asymptotic formula \eqref{asymp2} was established with the error $O_{\epsilon}(q^{-1/12+\epsilon})$.

Note that for weight $k>2$ the remainder term in  \eqref{asymp2} can be majorated by $O_{\epsilon,k}\left(q^{-1/4+\epsilon}\right)$. This was proved in \cite{B} for the case of prime power level $q=p^{\nu}$, $\nu>2$.

Another consequence of theorem \ref{mainthm} is a new subconvexity bound in the level aspect.
\begin{cor}
For all $\epsilon>0$
\begin{equation}
L(f,1/2+\mu)\ll_{\epsilon,\mu} q^{1/4-\delta},
\end{equation}
where $\delta=\frac{2\theta-1}{16(8\theta-7)}$.
\end{cor}
Taking $\theta=7/64$, we obtain $$\delta=\frac{25}{3136}=\frac{1}{125.44}.$$ The previously known result with
$\delta=1/192$ was established by Duke, Friedlander and Iwaniec \cite{DFI}.


\section{Selberg's eigenvalue conjecture}

Let $\Gamma$ be a congruence subgroup of modular group.
Let $0=\lambda _0< \lambda_1\leq \lambda_2 \leq \lambda_3 \leq \ldots$ be the eigenvalues of the automorphic Laplacian on $L^2(\Gamma\setminus \mathbb{H})$ induced from the Laplace operator
\begin{equation}
\Delta_L=-y^2\left( \frac{\partial^2}{\partial x^2}+\frac{\partial^2}{\partial y^2}\right).
\end{equation}
The eigenvalue $0<\lambda<1/4$ is called an exceptional eigenvalue.
\begin{conj}\label{SelConj} (Selberg, \cite{S})
The Laplacian for a congruence subgroup has no exceptional eigenvalues, i.e. $\lambda_1 \geq 1/4$.
\end{conj}

Below we provide several results related to  conjecture \ref{SelConj}:
\begin{itemize}
\item
$1965$ Selberg \cite{S}: $\lambda_1 \geq 3/16$;
\item
$1978$ Jacquet and Gelbart \cite{GJ}: $\lambda_1> 3/16$;
\item
$1995$ Luo, Rudnick, Sarnak \cite{LRS}: $\lambda_1>171/784$;
\item
$1996$ Iwaniec \cite{I}: $\lambda_1>10/49$;
\item
$2002$ Kim, Shahidi \cite{KSh}: $\lambda_1\geq 66/289$;
\item
$ 2003$ Kim, Sarnak \cite{KS}: $\lambda_1\geq  975/4096$.

\end{itemize}

Using the bound of Kim-Sarnak and equation \eqref{eq:theta}, we find
\begin{equation}
\theta=\sqrt{\max{(0,1/4-\lambda_1)}}=7/64.
\end{equation}


\section{Large sieve inequality}
Let $S(m,n;c)$ be the classical Kloosterman sum.
\begin{thm}\label{thm:DesIw}(theorem $9$ of \cite{DI} and lemma $9$ of \cite{M})
Let $r$, $s$ and $d$ be positive pairwise coprime integers with $r$ and $s$ square-free.
 Let $C$, $M$, $N$ be positive real numbers and $g$ be real-valued infinitely differentiable function with support in $[M,2M]\times [N,2N] \times [C,2C]$ such that
\begin{equation}\label{LSEcond}
\left| \frac{\partial^{(j+k+l)}}{\partial m^{(j)} \partial n^{(k)} \partial c^{(l)}}g(m,n,c)\right| \leq M^{-j}N^{-k}C^{-l} \text { for }0 \leq j,k,l \leq 2.
\end{equation}
Let $$X_d:=\frac{\sqrt{dMN}}{sC\sqrt{r}}.$$
Then for any $\epsilon>0$ and complex sequences $\uple{a}=\{a_m\}$, $\uple{b}=\{b_n\}$ one has
\begin{multline}
 \sum_{m}a_m\sum_{n}b_n \sum_{\substack{c\\ (c,r)=1}}g(m,n,c) S(dm\bar{r},\pm n;sc)  \ll_{\epsilon}\\
C^{\epsilon}
 d^{\theta}sC\sqrt{r}\frac{(1+X_{d}^{-1})^{2\theta}}{1+X_d}\left( 1+X_d+\sqrt{\frac{M}{rs}}\right)\left( 1+X_d+\sqrt{\frac{N}{rs}}\right)\\
 \times \left(\sum_{M< m\leq 2M }|a_m|^2\right)^{1/2} \left(\sum_{N< n\leq 2N}|b_n|^2\right)^{1/2},
\end{multline}
where $\theta$ is defined by equation \eqref{eq:theta}.
\end{thm}

\section{Error terms}
In this section, we consider the terms that give the largest contribution to the error in \cite{KMV}.  Our goal is to optimize the estimates of these terms and compute the exact dependence of the error on parameter $\theta$.

First, we improve bound $(21)$ of \cite{KMV}.

Note that the function $F_{M,N}(m,n)$ defined on page $108$ of \cite{KMV} is compactly supported on $[M/2,3M]\times [N/2,3N]$ and
\begin{equation}
F_{M,N}(x,y)\ll (1+|\mu|)^B(MN)^{-1/2}.
\end{equation}
\begin{lem}\label{lemma:LSI}
 Assume that  for any $\epsilon>0$ one has $M,N \ll q^{1+\epsilon}$. Then for any $C>\sqrt{lMN}$
 \begin{multline}\label{choice1}
\sum_{de=l}\frac{1}{d^{1/2}}\sum_{ab=d}\frac{\mu(a)}{a^{1/2}}\tau(b)
\sum_{\substack{c \geq C\\ q|c}}\frac{1}{c^2}T_{M,N}(c) \ll_{\epsilon} \\ (1+|\mu|)^B (Cq)^{\epsilon} l^{1/2}\left(\frac{\sqrt{MN}}{C}\right)^{1-2\theta}.
\end{multline}
\end{lem}
\begin{proof}
We split  $[C,\infty)$ into dyadic intervals and take $c \in [C,2C]$.
By equation $(18)$ of \cite{KMV} we have
\begin{multline*}
\sum_{q|c}\frac{1}{c^2}T_{M,N}(c)=
\sum_{n,m}\sum_{q|c}\tau(m)\tau(n)\frac{1}{c}S(m,aen;c)\\ \times J_{1}\left(\frac{4\pi\sqrt{aemn}}{c}\right )F_{M,N}(m,n)
= \frac{1}{q} \sum_{n,m}\tau(m)\tau(n)\\ \times\sum_{c_1}\frac{1}{c_1}S(m,aen,c_1q)J_{1}\left(\frac{4\pi\sqrt{aemn}}{c_1q}\right )F_{M,N}(m,n).
\end{multline*}
Here $m \in [M/2,3M]$, $n \in [N/2,3N]$ and $c_1 \in [C_1,2C_1]$ with $C_1:=C/q$.
Let $$Y:=\sqrt{MN}C_1\left(\frac{\sqrt{aeMN}}{C}\right)^{-1}.$$
As a test function we choose
$$g(m,n,c_1):=\frac{Y}{c_1}F_{M,N}(m,n)J_{1}\left(\frac{4\pi\sqrt{aemn}}{c_1q}\right).$$
It satisfies condition \eqref{LSEcond},
 and theorem  \ref{thm:DesIw} can be applied with $d=ae$, $r=1$ and $s=q$.
Hence
\begin{multline*}
\sum_{de=l}\frac{1}{d^{1/2}}\sum_{ab=d}\frac{\mu(a)}{a^{1/2}}\tau(b)\sum_{\substack{q|c\\ c \geq C}}\frac{1}{c^2}T_{M,N}(c) \ll_{\epsilon}  \\ (1+|\mu|)^B(Cq)^{\epsilon}l^{1/2} \left(\frac{\sqrt{MN}}{C}\right)^{1-2\theta}.
\end{multline*}

\end{proof}

The optimal value of $C$ can be chosen by making equal the estimate \eqref{choice1} and
the first summand of equation $(26)$ of \cite{KMV}, namely
\begin{equation}
l^{1/2}\left(\frac{\sqrt{MN}}{C}\right)^{1-2\theta}=l^{3/4}\frac{N^{1/4}}{M^{1/2}}\frac{C}{q}.
\end{equation}
This gives
\begin{equation}\label{eq:C}
C=l^{-\frac{1}{8-8\theta}}\min{\left(q^{\frac{1}{2-2\theta}}\sqrt{M}N^{\frac{1-4\theta}{8-8\theta}}, q^{\frac{9-8\theta}{8-8\theta}}  \right)}.
\end{equation}

After performing the dyadic summation over $M$ and $N$, we find that  for any $l<q^{\frac{1}{5-4\theta}}$ the error term in lemma \ref{lemma:LSI}  is bounded by
\begin{equation}
O_{\epsilon}\left(q^{\epsilon}(1+|\mu|)^Bl^{\frac{5-6\theta}{8-8\theta}}q^{-\frac{1-2\theta}{8-8\theta}} \right).
\end{equation}

Now we consider two other error terms that depend on $C$.
These are the errors resulting from extension of summation over $c>C$. See section $3.5$ (pages $111-112$) of \cite{KMV}.

Let
\begin{equation}
\eta_C(c):=\begin{cases}
1 & c \leq C\\
0 & otherwise
\end{cases}.
\end{equation}

\begin{lem}\label{error2}
Let $C$ be defined by equation \eqref{eq:C}. For any $\epsilon>0$
\begin{multline}
\sum_{M,N \ll q^{1+\epsilon}}\sum_{de=l}\frac{1}{d^{1/2}}\sum_{ab=d}\frac{\mu(a)}{a^{1/2}}\tau(b)
\sum_{q|c}(1-\eta_C(c))c^{-2}T^{OD}
\ll_{\epsilon} \\ (1+|\mu|)^B q^{\epsilon} l^{(5-4\theta)/(8-8\theta)}q^{-1/(8-8\theta)}.
\end{multline}

\end{lem}
\begin{proof}
Consider
\begin{multline*}
\sum_{q|c}(1-\eta_C(c))c^{-2}T^{OD}=
-2\pi \sum_{n}\tau(aen)\tau(n)\\ \times
\int_{0}^{\infty}Y_0(4\pi\sqrt{aent})J_1(4\pi\sqrt{aent})\sum_{\substack{q|c\\c>C}}\phi(c)F_{M,N}(c^2t,n)dt.
\end{multline*}
Since $C^2t<c^2t \leq 2M$, the sum over $c$ can be estimated as follows
\begin{equation*}
\sum_{\substack{q|c\\c>C}}\phi(c)F_{M,N}(c^2t,n)\ll\frac{1}{\sqrt{MN}}\frac{M}{qt}.
\end{equation*}
Next we apply $Y_0(x)\ll \log{x}$ and $J_1(x)\ll x$. Then
\begin{multline*}
\sum_{q|c}(1-\eta_C(c))c^{-2}T^{OD}
\ll_{\epsilon} \\(1+|\mu|)^B q^{\epsilon}\frac{N}{\sqrt{MN}}\int_{0}^{2M/C^2} t^{\epsilon}\frac{M}{qt} (aeNt)^{1/2}dt
\ll_{\epsilon} \\ (1+|\mu|)^B q^{\epsilon}(ae)^{1/2}\frac{MN}{qC}.
\end{multline*}

Finally, using \eqref{eq:C}, we obtain
\begin{multline*}
\sum_{M,N \ll q^{1+\epsilon}}\sum_{de=l}\frac{1}{d^{1/2}}\sum_{ab=d}\frac{\mu(a)}{a^{1/2}}\tau(b)
\sum_{q|c}(1-\eta_C(c))c^{-2}T^{OD}
\ll_{\epsilon} \\ (1+|\mu|)^B q^{\epsilon} l^{(5-4\theta)/(8-8\theta)}q^{-1/(8-8\theta)}.
\end{multline*}
\end{proof}

\begin{lem}\label{error3}
Let $C$ be defined by equation \eqref{eq:C}. For any $\epsilon>0$
\begin{multline}
\sum_{M,N \ll q^{1+\epsilon}} \sum_{de=l}\frac{1}{d^{1/2}}\sum_{ab=d}\frac{\mu(a)}{a^{1/2}}\tau(b)
\sum_{q|c}(1-\eta_C(c))c^{-2}T^{OOD}
\ll_{\epsilon} \\ (1+|\mu|)^B q^{\epsilon} l^{(5-4\theta)/(8-8\theta)}q^{-1/(8-8\theta)}.
\end{multline}
\end{lem}
\begin{proof}
According to \cite{KMV} page $111$ we have
\begin{equation*}
\sum_{q|c}(1-\eta_C(c))c^{-2}T^{OOD}\ll_{\epsilon} (1+|\mu|)^B q^{\epsilon}(ae)^{1/2}\frac{MN}{qC}.
\end{equation*}
Equation \eqref{eq:C}  yields the assertion.

\end{proof}

To sum up, the largest error terms in theorem \ref{mainthm} come from lemmas \ref{lemma:LSI}, \ref{error2}, \ref{error3} and equation $(26)$ of \cite{KMV}. In particular, the error term $O_{\epsilon}((1+|\mu|)^Bl^{17/8}q^{-1/4+\epsilon})$ is given by the second summand in $(26)$ of \cite{KMV}.


\section{Amplification and subconvexity}\label{amplifier}

Contribution of the main terms $M^{D}$,  $M^{OD}$,  $M^{OOD}$  in \cite{KMV} is  bounded by
\begin{equation}
O_{\epsilon}\left( (1+|\mu|)^Bq^{\epsilon}l^{-1/2}\right).
\end{equation}
According to theorem \ref{mainthm}, for $l<q^{\frac{1}{12-11\theta}}$ we have
\begin{multline}
\sum_{f \in H_{2}^{*}(q)}\frac{1}{4\pi \langle f,f\rangle_q} \lambda_f(l)|L(f,1/2+\mu)|^4\ll_{\epsilon, \mu} \\ q^{\epsilon} \left(l^{-1/2}+l^{\frac{5-6\theta}{8-8\theta}}q^{-\frac{1-2\theta}{8-8\theta}}\right).
\end{multline}

Let
\begin{equation}\label{eq:ampl}
\Lambda_f(\mathbf{c}):=\sum_{\substack{l \leq L\\ (l,q)=1}}c_l\lambda_f(l)
\end{equation}
be an amplifier. Then
\begin{multline}
\sum_{f \in H_{2}^{*}(q)}\frac{1}{4\pi \langle f,f\rangle_q} \Lambda_{f}^{2}(\mathbf{c})|L(f,1/2+\mu)|^4\ll_{\epsilon, \mu} \\
q^{\epsilon}\left(\|\mathbf{c}\|_{2}^{2}+L^{\frac{5-6\theta}{4-4\theta}}q^{-\frac{1-2\theta}{8-8\theta}} \|\mathbf{c}\|_{1}^{2}\right),
\end{multline}
where $\|\mathbf{c}\|_p$ denotes $l_p$--norm.

We choose coefficients $c_l$ as in \cite{DFI}, making $\Lambda_f(\mathbf{c})$ large for a particular form $f \in H_{2}^{*}(q)$, namely
\begin{equation}
c_{l}=\begin{cases}
\lambda_f(l) & \text{ if } l\text{ is prime } \leq L^{1/2}\\
-1 & \text{ if }l \text{ is a square of a prime} \leq L^{1/2}\\
0& \text{otherwise}.
\end{cases}
\end{equation}
Thus,
\begin{equation}
\Lambda_f(\mathbf{c})=\sum_{\substack{l \text{ prime } \leq L^{1/2}\\ (l,q)=1}}(\lambda_f(l)^2-\lambda_f(l^2)).
\end{equation}
Note that $\lambda_f(l)^2-\lambda_f(l^2)=1$ for prime $l$ such that $(l,q)=1$. Therefore,
\begin{equation}\label{eq:ampl1}
\Lambda_f(\mathbf{c})\sim 2 L^{1/2}(\log{L})^{-1}.
\end{equation}
By Deligne's bound
\begin{equation}\label{eq:ampl2}
\|\mathbf{c}\|_{2}^{2}\leq 5 \Lambda_f(\mathbf{c}) \text{ and } \|\mathbf{c}\|_{1}\leq 3 \Lambda_f(\mathbf{c}).
\end{equation}

The results of \cite{GHL} imply that
\begin{equation}\label{eq:ampl3}
\frac{1}{4\pi \langle f,f\rangle_q}\ll \frac{\log{q}}{q}.
\end{equation}

Taking $L=q^{\frac{2\theta-1}{2(8\theta-7)}}$ in  \eqref{eq:ampl} and applying  \eqref{eq:ampl1}, \eqref{eq:ampl2}, \eqref{eq:ampl3},   we have
\begin{equation}
L(f, 1/2+\mu)\ll_{\epsilon, \mu}\ q^{1/4-\delta}
\end{equation}
with $\delta=\frac{2\theta-1}{16(8\theta-7)}$.

\section{Mollification and simultaneous non-vanishing}

We follow section $5.2$ of \cite{KMV}. In order to determine the largest admissible length of mollifier $\Delta$, we sum the error terms in theorem \ref{mainthm} against $l^{-1/2+\epsilon}$ for $l<q^{2\Delta}$. This gives
\begin{equation}
q^{-\frac{1-2\theta}{8-8\theta}}q^{2\Delta\left( \frac{1-2\theta}{8-8\theta}+1\right)+\epsilon}+q^{-1/4}q^{21\Delta/4+\epsilon}+q^{-\frac{1}{8-8\theta}}q^{2\Delta\left(\frac{1}{8-8\theta}+1\right)+\epsilon}.
\end{equation}
Therefore, the error term is negligible for any $\Delta<\frac{1-2\theta}{2(9-10\theta)}$.

In order to change the harmonic mean into the natural average as defined by \eqref{eq:averages}, we apply results of section \ref{amplifier}.
Accordingly, condition $(82)$ of \cite{KMV} is satisfied for any $\Delta<\frac{1-2\theta}{4(7-8\theta)}$.

\begin{thm}
Let $M(f)$ be the mollifier defined by equation $(63)$ of \cite{KMV} with $P(x)=x^3$.  Let $F(\Delta)$ be defined by equation $(5)$ of \cite{KMV}.

For all
$0<\Delta_1<\frac{1-2\theta}{2(9-10\theta)}$ we have
\begin{equation}
\sum_{f \in H_{2}^{*}(q)}^{h}L(f,1/2)^4M(f)^4=(1+o(1))F(\Delta_1)\left(\frac{\zeta(2)}{\log{q}} \right)^4.
\end{equation}

For all
$0<\Delta_2<\frac{1-2\theta}{4(7-8\theta)}$ we have
\begin{equation}
\sum_{f \in H_{2}^{*}(q)}^{n}L(f,1/2)^4M(f)^4=(1+o(1))F(\Delta_2)\left(\frac{\zeta(2)}{\log{q}} \right)^4.
\end{equation}

\end{thm}

Taking $\theta=7/64$, we find that $\Delta_1<\frac{25}{566}=\frac{1}{22.64}$ and $\Delta_2<\frac{25}{784}=\frac{1}{31.36}$. This improves $\Delta_1<\frac{1}{30}$ and $\Delta_2<\frac{1}{48}$ proved in \cite{KMV}.

In particular, extention of admissible length of mollifier $\Delta_2$ gives a better lower bound on the  proportion of simultaneous non-vanishing
\begin{equation}
\sum_{\substack{f \in H_{2}^{*}(q)\\ L(f,1/2)L(f \otimes \chi, 1/2)\neq 0}} 1,
\end{equation}
where $\chi$ is a fixed primitive character of conductor $D$ such that $(D,q)=1$.
See Proposition $7.2$ of \cite{KMV} for the exact formulas.

\nocite{}

\end{document}